\documentclass[letterpaper,11pt,reqno]{amsart}
\makeatletter

\usepackage{amssymb}
\usepackage{latexsym}
\usepackage{amsbsy}
\usepackage{amsfonts}
\usepackage{hyperref}
\usepackage{graphicx}
\usepackage{enumerate}
\usepackage{color}

\usepackage{tikz}

\def\marginpar#1{\ignorespaces}

\newtheorem{theorem}{Theorem}[section]
\newtheorem{lemma}[theorem]{Lemma}
\newtheorem{proposition}[theorem]{Proposition}
\newtheorem{corollary}[theorem]{Corollary}

\numberwithin{equation}{section}
\makeatother

\newcommand{\hatQ}{\widehat{Q}}
\newcommand{\Qstar}{Q^*}
\newcommand{\FY}{F}
\newcommand{\hatgam}{\widehat{\gamma}}

\newcommand{\eps}{\varepsilon}
\newcommand{\hateps}{\widehat{\eps}}
\newcommand{\tileps}{\widetilde{\eps}}
\newcommand{\hatG}{\widehat{G}}
\newcommand{\Qb}{{Q}_\bullet}
\newcommand{\hatq}{\widehat{q}}
\newcommand{\mulog}{\mu_{\log}}
\newcommand{\mum}[2]{{ \mu _{#1,#2}   }}
\newcommand{\init}[1]{{ p_{0,#1} }}
\newcommand{\initnoarg}{{ p_{0,\bullet} }}

\newcommand{\BZ}{\mathbb{Z}}

\newcommand{\BP}{\mathbb{P}}
\renewcommand{\P}{\mathbb{P}}
\newcommand{\BE}{\mathbb{E}}
\newcommand{\E}{\mathbb{E}}

\newcommand{\T}{Y}
\newcommand{\St}{S}

\newcommand{\ed}{\overset{(d)}{{}={}}}
\newcommand{\eqth}{\overset{\theta}{{}={}}}

\newcommand{\convd}{\overset{(d)}{{}\rightarrow{}}}

\newcommand{\ind}{\mathbf{1}}
\newcommand{\Pbul}{P_{\bullet}}
\newcommand{\Pb}{P_{\bullet}}
\newcommand{\Nb}{N_{\bullet}}

\newcommand{\hatN}{\widehat{N}}
\newcommand{\Gb}{G_{\bullet}}

\newcommand{\giv}{\,|\,}  

\newcommand{\tilH}{\widetilde{H}}

\begin{document}
\title[Gaps and interleaving of point processes in sampling]{Gaps and interleaving of point processes \\ in sampling from a residual allocation model} 

\author[Jim Pitman]{{Jim} Pitman}
\address{Statistics department, University of California, Berkeley.  
} \email{pitman@stat.berkeley.edu}


\author[Yuri Yakubovich]{{Yuri} Yakubovich}
\address{St. Petersburg State University, Russia.
} \email{y.yakubovich@spbu.ru}

\date{\today} 
\begin{abstract}
This article presents a limit theorem for the gaps $\hatG_{i:n}:= X_{n-i+1:n} - X_{n-i:n}$ between order statistics $X_{1:n} \le \cdots \le X_{n:n}$ of a sample  of size $n$ from a random discrete distribution
on the positive integers $(P_1, P_2, \ldots)$ governed by a residual  allocation model (also called a Bernoulli sieve) 
$P_j:= H_j \prod_{i=1}^{j-1}(1-H_i)$ for a sequence of independent random hazard variables
$H_i$ which are identically distributed according to some distribution of $H \in (0,1)$ such that $- \log(1 - H)$ has a non-lattice distribution with finite mean $\mulog$.
As $n\to \infty$ the finite dimensional distributions of the gaps $\hatG_{i:n}$ converge to those of limiting gaps $G_i$ which are the numbers of points in a stationary renewal process with i.i.d.\ 
spacings $- \log(1 - H_j)$ between times $T_{i-1}$ and $T_i$ of births in a Yule process, that is $T_i := \sum_{k=1}^i \eps_{k}/k$ for a sequence of i.i.d.\ exponential variables $\eps_k$ 
with mean 1.  A consequence is that the mean of $\hatG_{i:n}$ converges to the mean of $G_i$, which is $1/(i \mulog )$.
This limit theorem simplifies and extends a result of Gnedin, Iksanov and Roesler for the Bernoulli sieve.
\end{abstract}

\maketitle


\setcounter{tocdepth}{1}
\tableofcontents

\section{Introduction}

Let $P_\bullet:= (P_1,P_2, \ldots)$ be a random discrete distribution on the positive integers described by a 
{\em residual allocation} or {\em stick-breaking} model (RAM), 
\cite{MR0010342} 
\cite{sawyer1985sampling} 
\cite{MR0411097}, 
also known as a {\em Bernoulli sieve} 
\cite{MR2044594} 
\cite{MR2735350} 
in which 
\begin{equation}
\label{stickbreak}
P_j:= H_j \prod_{i=1}^{j-1} ( 1 - H_i)
\end{equation}
where the $H_i \in (0,1)$, called {\em residual fractions}, {\em random discrete hazards}, or {\em factors}, are 
independent and identically distributed (i.i.d.)\ according to some distribution of $H \in (0,1)$.
Given $\Pb$, let $X_1,\dots, X_n$ be an i.i.d.\  sample of size $n$ from $\Pb$, and 
define counting variables 
\begin{equation}
\label{boxcounts}
N_{b:n} := \sum_{i=1}^n \ind ( X_i = b)  \qquad (b =1,2,\dots) .
\end{equation}
Study of the {\em reversed counts}, coming back from the maximum of the sample,
\begin{equation}
\label{hatn}
\hatN_{\bullet:n}:= (N_{M_n-i:n}, i = 0,1, \ldots), \mbox{ where  } M_n:= \max \{ b : N_b >0 \} = \max_{1 \le i \le n} X_i
\end{equation}
has been motivated by a number applications in ecology, population biology, and computer science 
\cite{MR0411097} 
\cite{sawyer1985sampling} 
\cite{MR827330} 
\cite{MR2735350} 
\cite{pitman-yakubovich-gem}. 
A model of particular interest in these applications is the 
{\em GEM$(0,\theta)$ model}, when it is assumed that $H$ has the beta$(1,\theta)$ distribution with density $\theta (1-u)^{\theta-1}$ at $u \in (0,1)$.
To provide some notation, let
\begin{equation}
\label{muij}
\mum{i}{j}:= \E H^i ( 1 - H)^j \eqth  \frac{ ( 1)_i (\theta)_j }{ (1 + \theta)_{i+j} }
\end{equation}
where $\eqth$ indicates an evaluation for the GEM$(0,\theta)$ model with $\P(H>u) \eqth (1-u)^\theta$, 
and 
$$(x)_i:= x ( x+1 ) \cdots (x+i-1) = \frac{ \Gamma( x + i)} {\Gamma(x)}$$
is 
Pochhammer's rising factorial function.
Let
\begin{equation}
\label{eq:mulog}
\mulog:= \E - \log ( 1 - H) = \sum_{m = 1}^\infty \frac{ \mum{m}{0}}{ m} = \int_0^1 \frac{ \BP( H > u ) }{1-u} \, du  \eqth \frac{1}{\theta}.
\end{equation}
The starting point of this article is the following limit theorem:
\begin{theorem}{\em (Gnedin, Iksanov and Roesler \cite[Theorem 2.1]{MR2508790})}. 
\label{thm:gir}
If
\begin{equation}
\label{nonlat}
\mbox{$- \log(1 - H)$ has a non-lattice distribution with finite mean $\mulog$}
\end{equation}
then
\begin{equation}
\label{gir:conv}
\hatN_{\bullet:n} \convd N_{\bullet} \mbox{ as } n \to \infty,
\end{equation}
 meaning that the finite dimensional distributions of the sequence $\hatN_{\bullet:n}$ converge to those of a limit sequence 
$N_\bullet$, which are determined by the formula
\begin{equation}
\label{gir:form}
\P(N_i = n_i, 0 \le i \le k) = \frac{ (n_0 + \cdots + n_k - 1)! }{n_0! \cdots n_k! } \frac{ \prod_{i = 0}^{k} \mum{n_i}{n_0 + \cdots + n_{i-1} } } {\mulog }
\end{equation}
for every finite sequence of non-negative integers $(n_0, \ldots, n_k)$ with $n_0 >0$.
\end{theorem}
Formula \eqref{gir:form} corrects a confusing mistake in the corresponding formula of \cite[Theorem 2.1]{MR2508790}. 
The proof of this result 
in \cite{MR2508790} 
involved a construction of $N_\bullet$ from a limiting point process, recalled later in \eqref{nfromxi}.
This article offers a simpler approach to an extended version of Theorem \ref{thm:gir}, with a simpler point process representation of the limiting sequence $\Nb$, 
and some corollaries which clarify special properties of the limit for the GEM$(0,\theta)$ model.

Let
\begin{equation}
\label{eq:defmax}
X_{1:n}:= \min_{1 \le i \le n} X_i \le X_{1:n} \le X_{2:n} \le  \cdots \le X_{n:n} = M_n:= \max_{1 \le i \le n} X_i ,
\end{equation}
denote the order statistics of the sample $X_1, \ldots, X_n$ from the RAM $\Pb$.
It was recently shown in  \cite{pitman-yakubovich-gem} 
that most known properties of GEM$(0,\theta)$ samples are consequences of the following simple description 
 of the {\em sequence of gaps} $\hatG_{\bullet:n}$ between order statistics, in reversed order
\begin{equation}
\label{gaps}
\hatG_{i:n} := X_{n+1-i:n} - X_{n-i:n}  \qquad (1 \le i \le n)
\end{equation}
with the convention $X_{0:n}:= 1$, so that $\hatG_{n:n}:= X_{1:n} - 1$. According to the main result of 
\cite{pitman-yakubovich-gem} 
\begin{quote}
{\em
in sampling from GEM$(0,\theta)$ the $\hatG_{i:n}$ are independent, with geometric distributions whose parameters depend only on $i$ and $\theta$, and not on the sample size $n \ge i$, according to the equality in distribution
\begin{equation}
\label{gapstheta}
(\hatG_{i:n} , 1 \le i \le n) \ed  (G_i, 1 \le i \le n)
\end{equation}
where $G_\bullet := (G_i, i = 1,2, \ldots)$ is an infinite sequence of independent random variables such that
$G_i$ has the geometric$( i/(i+\theta))$ distribution on $\{0,1, \ldots\}$ with $\E G_i \eqth \theta/i$ and 
\begin{equation}
\label{gapsthetageo}
\P(G_i \ge k ) \eqth \left( \frac{\theta}{i+\theta}\right)^k \qquad ( k = 0,1,2, \ldots).
\end{equation}
}
\end{quote}
It is an obvious consequence of \eqref{gapstheta} that $\hatG_{\bullet:n} \convd \Gb$ in the GEM$(0,\theta)$ model.
The aim of the present study is
to relate this limit sequence of reversed gaps $\Gb$ 
to the limiting reversed counts $\Nb$ in \eqref{gir:conv}, first of all for the GEM$(0,\theta)$ model,
then to describe a corresponding limit sequence of reversed gaps $\Gb$ for a more general RAM.

In sampling from any random discrete distribution $\Pb$,
the two sequences $\hatG_{\bullet:n}$ and $\hatN_{\bullet:n}$ are related 
via the {\em time-reversed tail count sequence} $\hatQ_{\bullet:n}$ defined by the partial sums of $\hatN_{\bullet:n}$.
So, using the well-known balls in the box description of the sampling procedure, recalled in Section \ref{sec:starsbars} below,
\begin{equation}
\label{qfromn}
\hatQ_{k:n}:= \sum_{i=0}^k \hatN_{i:n} := \sum_{i = 0}^k N_{M_n-i:n} \qquad( k = 0,1, 2 , \ldots)
\end{equation}
is the number of balls in the last $k+1$  boxes when $n$ balls are distributed, 
counting back from the rightmost occupied box. The count of balls in that box is the initial reversed tail count,
indexed by $k = 0$:
\begin{equation}
\label{qfromnzero}
\hatQ_{0:n}:= \hatN_{0:n} := N_{M_n:n}  := \sum_{i = 1}^n \ind ( X_i = M_n) > 0.
\end{equation}
For each $g \ge 0$, a value $j < n $ appears exactly $g$ times in the sequence $\hatQ_{\bullet:n}$ iff  there is
a corresponding gap between order statistics $\hatG_{j:n} = g$. Thus 
\begin{equation}
\label{ghatfromq}
\hatG_{j:n} = \sum_{k=0}^n \ind ( \hatQ_{k,n} = j) \qquad ( 1 \le j < n ) .
\end{equation}

For a RAM subject to \eqref{nonlat}, the connection \eqref{qfromn}--\eqref{ghatfromq}
between the reversed gaps $\hatG_{\bullet:n}$ and  reversed counts $\hatN_{\bullet:n}$ shows that
convergence in distribution of $\hatG_{\bullet:n}$ to a limiting sequence of gaps $\Gb$ holds jointly with
the convergence in distribution \eqref{gir:conv}
of $\hatN_{\bullet:n}$ to $\Nb$, 
with $\Gb$ the sequence of occupation times of states by
the partial sums $\Qb$ of $\Nb$.
The conclusion of this argument is expressed by the following extension of Theorem \ref{thm:gir}.

\begin{theorem} 
\label{thm:main}
Suppose the common distribution of factors $H$ in the RAM, with  moments \eqref{muij}, 
 is such that $- \log ( 1 - H) $ is non-lattice with finite mean $\mulog$.
Then as $n \to \infty$ there is the joint convergence of finite-dimensional distributions
\begin{equation}
\label{nglims} 
(\hatQ_{\bullet:n}, \hatN_{\bullet:n}, \hatG_{\bullet:n} ) \convd (\Qb, \Nb, \Gb)
\end{equation}
with limit sequences $\Qb$ and $\Nb$ indexed by $k = 0,1,2, \ldots$, and $\Gb$ indexed by $j = 1,2, \ldots$.
The joint distribution of these three limit sequences is defined as follows:
\begin{itemize}
\item
The sequence $\Qb$ is a Markov chain with stationary transition probability matrix
\begin{equation}
\label{backfor}
p_{m,n}:=  \binom{n - 1 }{ m - 1 } \mum{ n- m }{m}   \qquad (1 \le m \le n )
\end{equation}
where $\mum{ n- m }{m}:= \E H^{n-m} (1-H)^m$, and initial distribution 
\begin{equation}
\label{entrancelaw}
\P( Q_0 = m ) =  \frac{ \mum{ m}{0} }{ m \, \mulog } \qquad ( m = 1, 2, \ldots )
\end{equation}
which is the limit distribution of $N_{M_n:n}$, the number of balls in the last occupied box,  as $n \to \infty$.
\item
The limit sequence of reversed counts $\Nb$ is the difference sequence of $\Qb$, with $N_0:= Q_0$ and
$N_i:= Q_{i} - Q_{i-1}$ for $i = 1,2, \ldots$.
\item
The limit sequence of reversed gaps $\Gb$ is the sequence of occupation counts of the Markov chain $\Qb$:
\begin{equation}
G_j := \sum_{k = 0}^\infty \ind ( Q_k = j )  \qquad (j = 1,2, \ldots ).
\end{equation}
\end{itemize}
\end{theorem}
It should be clear from the above discussion that Theorem \ref{thm:main} can easily be deduced from
Theorem \ref{thm:gir}, and vice versa. In particular, the formula \eqref{entrancelaw} for the initial
distribution of $Q_0 = N_0$ is the instance $k=0$ of the formula \eqref{gir:form} for the limiting distribution of
$N_{M_n:n}$, which was first found in \cite{MR2538083}. 
See also \cite[Theorem 6.1]{MR2735350}. 
The equivalence of formula \eqref{gir:form} for general $k \ge 0$ with the Markov property of $\Qb$ expressed in Theorem \ref{thm:main} is also easily checked.
Further study of the chain $\Qb$ in the GEM$(0,\theta)$ case leads to a remarkable connection to the Riemann zeta function \cite{zeta}. 
Theorem \ref{thm:main} will be proved in Section \ref{sec:proofs}, along with 
the following corollaries.

\begin{corollary}
\label{crl:g}
In the setting of Theorem \ref{thm:main}: 
\begin{itemize}
\item The distribution of $G_j$, 
the limit in distribution of $\hatG_{j:n}$ as $n \to \infty$, 
is the zero-modified geometric distribution
with parameters $(h_j, 1 - \mum{0}{j})$, meaning that
\begin{equation}
\label{zmodmum}
\P(G_j \ge k) =  h_j  \, \mum{0}{j}^{k-1}  \qquad ( j, k = 1,2, \ldots )
\end{equation}
where
\begin{equation}
\label{hmudef}
h_j: = \P(G_j \ge 1)  =  \frac{ 1 - \mum{0}{j} } { j \mulog } \qquad ( j = 1,2, \ldots).
\end{equation}
\item 
The conditional distribution of $G_j$ given $G_j \ge 1$ is the
geometric$( 1 - \mum{0}{j})$ distribution on $\{1,2, \ldots\}$. 
\item All moments of $G_j$ are finite,  in particular
\begin{equation}
\label{valposg}
\E G_j  = \frac{ h_j  }{1 - \mum{0}{j}}  = \frac{ 1 } {j \mulog} 
\end{equation}
and there is convergence of moments $\lim_{n\to\infty} \E \hatG_{j:n} ^p = \E G_j^p$ for every $p\ge 0$.
\end{itemize}
\end{corollary}

\begin{corollary}
\label{crl:gth}
In the setting of Theorem \ref{thm:main}, for a RAM with i.i.d.\ factors,
the following four conditions are equivalent:
\begin{itemize}
\item the model is GEM$(0,\theta)$, meaning that $H$ is beta$(1,\theta)$, for some $\theta >0$;
\item distribution of $G_j$ is geometric$(p_j)$ on $\{0,1, \ldots\}$ for some $p_j$, for all $j = 1,2, \ldots$;
\item the limiting gaps $G_j$ are independent random variables;
\item for the prelimit gaps the probability $\P(\hatG_{1:n}\ge 1)$ does not depend on $n$. 
\end{itemize}
Then $p_j=j/(j+\theta)$ and the formulas of Corollary \ref{crl:g}
hold with 
\begin{equation}
\label{thetavals}
\mulog \eqth \frac{1} {\theta}, 
\qquad h_j \eqth \mum{0}{j} \eqth \frac{ \theta }{ j + \theta }.
\end{equation}
\end{corollary}

As discussed in \cite{MR2508790}  
for a general RAM, and in 
\cite{pitman-yakubovich-gem} 
for GEM$(0,\theta)$,
the limit sequences $\Nb$ and $\Gb$ may be encoded in various ways as counts in suitable point processes.
But those studies overlooked the simple basic structure of the Markov chain $\Qb$ exposed by the next two corollaries.
\begin{corollary}
\label{crl:branching}
For any distribution of $H$ on $(0,1)$, formula \eqref{backfor} defines the transition mechanism  of the Markov chain $\Qb$ 
 defined by a Galton-Watson branching process in a random environment, in which at each generation $k$ the offspring distribution of all individuals present is geometric$(1-H_k)$ on $\{1,2, \ldots \}$, with the $H_k$ picked i.i.d.\ from one generation to the next according to the distribution of $H$.  
Thus for each $k \ge 0$, by conditioning on $H_k$,
\begin{equation}
\label{negbingf}
\E \left(z ^{ Q_{k+1} - Q_k} \, \left| \, Q_k  = m \right.\right) = \E \left(  \frac{  1 - H }{ 1 - H z } \right)^m
\end{equation}
which when expanded in powers of $z$ is equivalent to \eqref{backfor}.
\end{corollary}

This is a mixed negative binomial$(m,1-H)$ distribution for the increment
$Q_{k+1} - Q_k$ given $Q_k  = m $. With the shift by $m$, the
conditional distribution of $Q_{k+1}$ given $Q_k  = m $ 
is the mixed Pascal$(m, 1-H)$ distribution of the waiting time until the  $m$th success in a series of trials, which given $H$
are i.i.d.\ with success probability $1-H$ per trial.
In particular, if $H$ is assigned  a beta$(a,b)$ distribution on $(0,1)$, the distribution of 
$Q_{k+1} - Q_k$ given $Q_k  = m $ is known as a {\em beta mixed negative binomial} or {\em inverse Polya-Eggenberger}
distribution  \cite[\S 8.4]{MR2163227}. 
Some basic properties of the chain $\Qb$, such as the formula \eqref{meanq} below for $\E(Q_k)$ and asymptotic growth of $Q_k$ for large $k$, can be read from 
general theory of branching processes in a random environment. See for instance
\cite{MR0246380} 
\cite{MR1174427} 
\cite{MR2118578}. 
But it seems easiest to obtain results about $Q_k$ for large $k$ from the following simple construction:

\begin{corollary}
\label{crl:yule}
For $k = 0,1,2, \ldots$ let $Y_k:= \sum_{i = 1}^k \eps_i/i$ where the $\eps_i$ are independent standard exponential variables, so that $0= Y_0 < Y_1 < Y_2 < \cdots$ are the times of births in a {\em standard Yule process} 
$N_Y(t):= \sum_{k=0}^\infty \ind (Y_k \le t), t \ge 0$. 
Independent of this Yule process, let
$0 < S^*_0 < S^*_1 < \cdots$ be the times of arrivals 
in 
a stationary renewal counting process $N^*_S(t):= \sum_{i=0}^\infty \ind (S^*_i \le t), t \ge 0$,  
with $S^*_i - S^*_{i-1}$ for $i \ge 1$ a sequence of i.i.d.\ copies of $-\log ( 1 - H)$, and $S^*_0$
independent of the $S^*_i - S^*_{i-1}$ for $i \ge 1$, with the stationary delay 
distribution
\begin{equation}
\label{sodist}
\P( S^*_0 \in ds ) 
= \frac{ \BP( - \log(1-H)  > s )  } { \mulog  } \, ds
= \frac{ \BP(  H > 1 - e^{-s} )   } { \mulog  } \, ds
\qquad ( s > 0 ) 
\end{equation}
where $\mulog:= \E - \log(1-H)$ as in \eqref{eq:mulog}.
Then the limiting reversed tail count Markov chain $\Qb$ may be constructed as 
\begin{equation}
\label{yulerep}
Q_k := N_Y ( S^*_k)   \qquad (k = 0,1,2, \ldots ) 
\end{equation}
along with a corresponding representation of the cumulative sums of gaps $\Gb$:
\begin{equation}
\label{yulerepg}
G_1 + \cdots + G_j  = N^*_S ( Y_j)  \qquad (j = 1,2, \ldots ).
\end{equation}
\end{corollary}

These results may be summarized less formally as follows.
Regard the sample from the RAM as a pattern of $n$ balls in an infinite row of boxes labeled by positive
integers, from left to right. Read the pattern from right to left, starting with the rightmost occupied box, 
as a list of $n$ stars and some number of bars, with a star for each ball, and each bar representing a barrier between boxes.
The pattern starts with $N_{0:n} \ge 1$ stars followed by a bar, then $N_{1:n} \ge 0$ stars, followed by a bar, and so on.
Then $G_{j:n}$ for $1 \le j < n$ is the number of bars between the $j$th and $(j+1)$th star.

As $n \to \infty$, the distribution of each initial finite segment of this pattern of stars and bars converges to that of an initial segment of an infinite pattern of stars and bars. The limit pattern
starts with a star and contains both infinitely many stars and infinitely many bars almost surely.
This limit pattern is generated by the superposition of two independent simple point processes on $[0,\infty)$, 
a Yule birth process of stars at times $0 = Y_0 < Y_1 < \cdots$, and  
a stationary renewal process of bars at times $0 < S^*_0 <  S^*_1 <  \cdots$, 
by listing the stars and bars in their order of appearance, always starting with a star from initial birth of the Yule process at time $0$.

Thus
the two limit sequences $\Nb$ and $\Gb$ provide dual encodings of the interleaving of points of the two simple point processes $S^*$ and $Y$ on $[0,\infty)$:
\begin{itemize}
\item $\Gb$ counts $S^*$ renewals between consecutive Yule birth times;
\item $\Nb$ counts Yule birth times between consecutive $S^*$ renewal points.
\end{itemize}
For instance, if the interleaving  of birth times and renewal points in increasing order is
$$
(0 = \T_0 , \T_1 ,\St^*_0 , \T_2 , \St^*_1 , \St^*_2 , \St^*_3 , \T_3 , \T_4 , \T_5 , \T_6 , \St^*_4 , \T_7 , \T_8 , \T_9 ,  \cdots )
$$
which information is fully encoded by the string 
$YYS{\,}YS{\hspace{1.2pt}S\hspace{1.2pt}}S{\,}YYYYS{\,}YYY \cdots $,
the two sequences describing this interleaving are
$$
(G_1 , G_2, \ldots) = (  0,1,3,0,0,0,1,0,0, \ldots )
\mbox{ and } 
(N_0, N_1, \ldots) = ( 2,1,0,0,4,3+ , \ldots ) 
$$
where the $3+$ indicates a value which is at least $3$, but which cannot be determined exactly without examining more points of the interleaving.
This is a limiting form of the classical ``stars and bars'' duality of combinatorics for enumerating various kinds of integer compositions.  

In the usual representation  of sampling from a RAM, by the locations of $n$ uniform sample  points in the unit interval, relative to a sequence of bins of lengths $P_1, P_2, \ldots$,
the two point processes $N_Y$ and $N^*_S$ arise from an asymptotic analysis on a logarithmic scale.
The structure of the sample points on that scale, coming down from the maximum, is that of the birth times of the Yule process,
with $Y_0 = 0$ corresponding to the maximum value in the sample of $n$ independent variables.
The two point processes which play dual roles in this asymptotic description are
\begin{align}
\label{statren} N^*_S(t) &:= \sum_{i = 0 }^\infty \ind ( S^*_i \le t )  \mbox{  has stationary increments with } \BE N^*_S(t) = t/ \mulog ; \\
\label{yulep} N_{Y}(t)&:= \sum_{i = 0 }^\infty \ind( Y_i \le t )  \mbox{ is a Markovian branching process with } \BE N_{Y}(t) = e^{t}. 
\end{align}
The GEM$(0,\theta)$ model is by far the simplest, because  of the result of Ignatov \cite{MR645134}:  
\begin{equation}
\label{ignatov}
\mbox{{\em  in the \textup{GEM}$(0,\theta)$ model $(N^*_S(t), t \ge 0 )$ is a Poisson process with constant rate $\theta$}.}
\end{equation}
See \cite{pitman-yakubovich-gem} 
for further analysis in this case.

\begin{corollary}
In the setting of Theorem \ref{thm:main}
for all $j \ge 0$
\begin{align}
\label{meanq} \lim_{n \to \infty} \BE \hatQ_{j:n} &=  \BE Q_j =  \frac{ (\mum{0}{-1} - 1) \, \mu^j_{0,-1} }{\mulog} 
\end{align}
where $\mu_{0,-1}:= \E (1- H)^{-1}$, along with a corresponding result for $\BE N_j$ by differencing, with $\BE Q_j = \BE N_j = \infty $ for all $j$ if $\mum{0}{-1}:= \BE ( 1 - H)^{-1}  = \infty$.
For the GEM$(0,\theta)$ model, $\mum{0}{-1} \eqth 1/(\theta - 1)_+$, with value $\infty$ iff $0 < \theta \le  1$.
Also,  in the setting of Theorem \ref{thm:main},
the Markov chain $\Qb$ is subject to the exponential growth
\begin{align}
\label{borel}
\lim_{k \to \infty} Q_k^{1/k}  = \exp(\mulog)  \mbox{ almost surely.}
\end{align}
\end{corollary}
The formula  \eqref{meanq} for $\E Q_j$ is read from the branching process description of $\Qb$,
and the convergence of expectations is easily justified, much as in  
\cite{MR2508790} 
for the limiting counts $K_j=\sum_{i=0}^\infty \ind (N_i=j)$.
The limit formula \eqref{borel} also follows easily from Corollary \ref{crl:yule}.
The particular case of \eqref{borel} for the Markov chain $\Qb$ associated with the GEM$(0,1)$ model with $\mulog = 1$ was first encountered by 
Borel \cite{MR0023007} 
and L\'evy \cite{MR0023008} 
in connection with the representation of a number between 0 and 1 in the so-called Engel series, see also 
\cite{MR0102496} for detailed proofs. 

\section{Preliminaries}

\subsection{Stars and bars}
\label{sec:starsbars}
The notion of a random sample from a random discrete distribution admits a variety of possible interpretations. 
See for instance 
\cite{pitman-yakubovich-gem} 
for recent review.
But the balls-in-boxes metaphor from  recent studies of the Bernoulli sieve \cite{MR2735350} 
seems to provide the most intuitive language for the present analysis.
Regard the sample  $X_1, \ldots, X_n$ as an allocation of $n$ balls labeled by $i = 1,2, \ldots, n $ into an unlimited number of boxes 
labeled by $b \in \{1,2, \ldots\}$. 
So $X_i$ is the label of the box into which ball $i$ is thrown. Given $\Pbul$ the $X_i$ are independent allocations with 
$\BP(X_i = b \giv \Pbul) = P_b$.  The count $N_{b:n}$ is the number of balls thrown into box $b$, the sample maximum $X_{n:n} = \max\{b: N_b >0\}$ 
is the label of the rightmost occupied box, and so on. 

Represent a configuration of balls in boxes by its classical combinatorial encoding as a list of {\em stars and bars} \cite[p.~15]{MR2868112}.  
For instance, the configurations of values and their multiplicities in two possible samples of size  $8$ from positive integers may be indicated as
\begin{equation}
N_{\bullet:8} = ( 2 , 0, 1, 2, 0, 3)   \qquad \mbox{ or }\qquad  ( 0,0,0, 2, 0, 3, 0, 2, 1)  
\end{equation}
which correspond to the stars and bars sequences
\begin{equation}
\label{starsbars}
{{}*{}} {{}*{}} |\,|{{}*{}} | {{}*{}}{{}*{}} |\, | {{}*{}}{{}*{}}{{}*{}}  |\, |\, |\, |\, | \cdots \qquad
\mbox{ or } \qquad
|\,|\,|{{}*{}}{{}*{}} |\, |{{}*{}}{{}*{}}{{}*{}} |\, | {{}*{}}{{}*{}} | {{}*{}} |\, |\, | \cdots \,.
\end{equation}
Here $N_{b:8}$ is the number of stars between the $(b-1)$th and $b$th bar, reading from left to right, and the sequence is terminated for convenience at the rightmost occupied box $b= M_8$,
 beyond which all the counts are $0$.
Each star represents a ball, and each bar a {\em barrier} between boxes.
Box 1 is the container to the left of the first bar, box 2 is between the first bar and the second bar, and so on. The $\cdots$ represent an unlimited number of additional boxes, all of which are 
empty in these configurations with only $8$ balls.
These configurations are just as well encoded by their order statistics, or by the gaps between their order statistics. 
The sequence of gaps between order statistics $\hatG_{\bullet:8}$,
defined from the top down as in \eqref{gaps},
counts numbers of bars between stars (barriers between balls), from right to left, starting with the rightmost ball:
\begin{equation}
\hatG_{\bullet:8} =  (0,0,2,0,1,2,0,0) \mbox{ or } (1,0,2,0,0,2,0,3) .
\end{equation}
These two configurations illustrate for $n = 8$ the cases
\begin{equation}
(N_{1:n} > 0 ) \Leftrightarrow (\hatG_{n:n} = 0 ) \mbox{ or } (N_{1:n} = 0 ) \Leftrightarrow (\hatG_{n:n} > 0 )  .
\end{equation}
These relations hold no matter what the sample, by 
the identity of events $(N_{1:n} >0) = (X_{1:n} = 1)$ and
the definition of $\hatG_{n:n}:= X_{1:n} - 1$.

For sampling from any random discrete distribution, the distribution of $N_{\bullet:n}$ is just a mixture of infinitinomial$(\Pbul)$ distributions, as treated in 
\cite{MR0216548} 
\cite{MR2318403} 
\cite{MR2412154} 
for a fixed sequence of probability parameters $\Pbul$ on the positive integers, with the parameter sequence $\Pbul$ assigned some probability distribution. Thus
for every sequence of integers $(n_1, \ldots, n_k)$ with $n_i \ge 0$ for $1 \le i < k$ and $n_k >0$ and $\sum_{i=1}^n n_i = n$,
\begin{align}
\label{distcounts} \BP [ N_{\bullet:n} = (n_1,\ldots, n_k, 0,0,\ldots) ] &= {n \choose n_1, \ldots, n_k } \BE \prod_{i=1}^k P_i ^{n_i}  \\
\label{sashaform}                   &= {n \choose n_1, \ldots, n_k }   \E \prod_{i=1}^ k H_i ^{n_i} ( 1 - H_i )^{n_{i+1} + \cdots + n_k } \\
\label{ramform} &= {n \choose n_1, \ldots, n_k }   \prod_{i=1}^ k \mum{n_i}{n_{i+1} + \cdots + n_k} 
\end{align}
where in the last expression it is assumed that $\Pbul$ follows a RAM with i.i.d.\ hazards $H_k \ed H$ with moments $\mum{i}{j}:= \BE[ H^i (1-H)^j]$. This expression for a RAM
corrects a formula of 
Gnedin, Iksanov and Roesler \cite[(3)]{MR2508790}, 
in which the multinomial coefficient should be omitted, and the order of indices reversed.
Formula  \eqref{distcounts} also gives 
\begin{equation}
\label{gapsdist} \BP [ \hatG_{\bullet:n} = (g_1, \ldots, g_n) ] \mbox{ for } (g_1, \ldots g_n) \leftrightarrow (n_1,\ldots, n_k, 0,0,\ldots) 
\end{equation}
via the bijection (\ref{qfromn})--(\ref{ghatfromq}) mediated by the stars and bars representation between  the
$n$-tuples of non-negative integers $(g_1, \ldots, g_n)$ which are the possible values of the gap sequence $\hatG_{\bullet:n}$ 
and the weak compositions $(n_1,\ldots, n_k, 0,0,\ldots)$ of $n$ with $n_j \ge 0$ and $\sum_j n_j = n$ which are possible values of the count sequence $N_{\bullet:n}$.
In principle, formula \eqref{distcounts} specifies the distribution of the gaps  $\hatG_{\bullet:n}$ for a sample of size $n$ from any random discrete distribution $\Pbul$, with 
some simplification for a RAM.  
The distribution of the gap sequence $\hatG_{\bullet:n}$ derived from  GEM$(0,\theta)$ is especially simple, 
as indicated in \eqref{gapstheta}, and far simpler than its logically equivalent description in terms of the
counts sequence $N_{\bullet:n}$, via the stars and bars bijection. 

\subsection{Point processes}
\label{sec:sampingram}

For a random discrete distribution $\Pbul$ on the positive integers, assumed to be of the stick-breaking form \eqref{stickbreak} for some
(possibly dependent) random factors $H_i \in [0,1]$, let
\begin{equation}
\label{breakpoints}
\FY_0:= 0 \mbox{ and } \FY_j:= \sum_{i=1}^j P_i = 1 - \prod_{i=1}^j (1 - H_i) \mbox{ for } j = 1,2, \ldots.
\end{equation}
In the stick-breaking interpretation, the $\FY_k \in [0,1]$ are called the {\em break points}. 
For a random sample $X_1,X_2, \ldots$ from $\Pbul$,
\begin{equation}
\label{breakcdf}
\BP( X_i \le j \giv \Pbul ) =  \FY_j\,,
\end{equation}
so the sequence $\FY_\bullet=(F_0,F_1,F_2,\dots)$ gives the evaluations at $j = 0,1, 2, \ldots$ of the random discrete cumulative distribution function derived from $\Pbul$.
Suppose that almost surely
\begin{equation}
\label{proper}
P_j > 0 \mbox{ and } \sum_j P_j = 1, \mbox{ or,  equivalently }  0 < \FY_1 < \FY_2 < \cdots \uparrow 1 .
\end{equation}
Following the method used by 
Ignatov \cite{MR645134} 
for GEM$(0,\theta)$,  and further developed 
by Gnedin and coauthors 
\cite{MR2508790} 
for various other models of random discrete distributions,
the  counts of break points $\FY_j$, not including either end of the interval $[0,1]$,
 define a simple point process $N_\FY$, with 
$$
N_\FY(a,b]:= \sum_{k=1}^\infty \ind( a < \FY_k  \le b )   \qquad ( 0 \le a < b < 1 )
$$
the number of break points in $(a,b]$.
It is convenient to make the change of variable from $[0,1)$ to $[0,\infty)$ by the map $u\mapsto x = - \log(1-u)$. 
This is the inverse of the cumulative distribution function $x \mapsto 1 - e^{-x}$ of a standard exponential variable 
$\eps := - \log (1-U)$ for $U$ uniform on $(0,1)$.
Let  $S_0:= 0$ and
$$
S_j:= - \log ( 1 - \FY_j) = \sum_{i=1}^j - \log(1 - H_i) \qquad ( j = 1,2, \ldots).
$$
Regard these images $S_j$ of the break points $\FY_j$ as the points of an associated point process $N_S$ on $(0,\infty)$, 
$$
N_S(s,t]:= \sum_{j=1}^\infty \ind( s < S_j \le t ) =  N_\FY(1 - e^{-s}, 1 - e^{-t}]   \qquad (0 \le s < t < \infty )
$$
and consider the {\em spacings} $S_j - S_{j-1}$ between these points. Note that the assumption \eqref{proper} translates 
into $0 = S_0 < S_1 < S_2 < \cdots \uparrow \infty$.
The following Lemma is obvious from these definitions.
\begin{lemma} $~$
\label{lemma:ram}
\begin{itemize}
\item [(i)]
In the representation of $\Pbul$ as a residual allocation model \eqref{stickbreak} or \eqref{breakpoints}, the factors $H_i$ are independent
iff the spacings $S_j - S_{j-1}$, $j = 1,2, \ldots$ are independent.
\item [(ii)]
The $H_i$ are i.i.d.\ iff the $S_j - S_{j-1}$ are i.i.d., in which case the $S_j$ are the points of a renewal process on $(0,\infty)$, with spacings 
distributed according to $S_1 \ed - \log( 1 - H_1)$.
\end{itemize}
\end{lemma}

The key simplifying property of the GEM$(0,\theta)$ model is Ignatov's result \cite{MR645134} 
that its associated renewal process is a homogeneous Poisson process with rate $\theta$.
To  analyse the process of sampling from a random discrete distribution $\Pbul$, the sample $X_1, X_2, \ldots$ may be constructed in the
usual way from the discrete random cumulative values $\FY_\bullet$ as 
\begin{equation}
\label{xidef}
X_i := 1 + N_\FY(0,U_i]:= 1 + \sum_{j=1}^\infty \ind(\FY_j \le U_i )  \qquad ( i = 1,2, \ldots)
\end{equation}
where $U_1, U_2, \ldots$ is a  sequence of i.i.d.\ uniform $(0,1)$ variables, independent of the break points $\FY_\bullet$ counted by $N_\FY$.
Here, the parts of the stick between break points are labeled from left to right by $1,2, \ldots$. Then $X_i$ is the label of the part of the stick containing $U_i$.
That label is conveniently evaluated in \eqref{xidef} as $1$ plus the number of break points to the left of $U_i$. 
The next Lemma too follows immediately from these definitions.

\begin{lemma}
\label{lmm:pp}
Whatever the random discrete distribution $\Pbul$ subject to \eqref{proper}, let the sample $(X_1, \ldots, X_n)$ be constructed by \eqref{xidef} 
from the break points $\FY_\bullet$ and an i.i.d.\ uniform sample  $(U_1, \ldots, U_n)$,  with $S_j:= - \log (1-\FY_j)$ the transformed
break points, and $\eps_i:= - \log(1- U_i)$ the transformed uniforms, which are an i.i.d.\ standard exponential sample. Then
the order statistics of the $X$-sample, the $U$-sample and the $\eps$-sample are related by
\begin{equation}
\label{sysgaps}
X_{i:n} = 1 + N_\FY(0,U_{i:n}] = 1 + N_S(0,\eps_{i:n}]\qquad\quad (1\le i\le n), 
\end{equation}
while the gaps of the $X$-sample in descending order $\hatG_{i:n} := X_{n-i+1:n} - X_{n-i,n}$ may be recovered from counts in either of the point processes $N_\FY$ or $N_S$ as
\begin{equation}
\label{gue}
\hatG_{i:n} = N_\FY( U_{n-i:n}, U_{n-i +1:n} ] =  N_S( \eps_{n-i:n}, \eps_{n-i +1:n} ] \qquad\quad (1\le i\le n).
\end{equation}
\end{lemma}

\subsection{The Yule process}

The following proposition presents a number of known characterizations of the Yule process.

\begin{proposition}
\label{prp:yule}
Let $Y_0:= 0 < Y_1 < Y_2 < \cdots$ be the points of a simple point process on $[0,\infty)$, with an initial
point at $0$, and associated counting process
$$
N_Y(t):= \sum_{k=0}^\infty \ind ( Y_k \le t ) \qquad ( t \ge 0).
$$
The following conditions are equivalent:
\begin{itemize}
\item The counting process $N_Y$ is a {\em standard Yule process}, that is a pure birth Markov process, with state space $\{1,2, \ldots\}$,
initial state $N_Y(0) = 1$, right continuous step function paths,  constant transition rate $k$ for transitions from  $k$ to $k+1$, and all other
transition rates $0$.
\item The counting process $N_Y$ is a Markov process with stationary transition probabilities, such that $N_Y(0) = 1$ and for $s, t \ge 0$, for each $m = 1,2, \ldots$ the
conditional distribution of $N_Y(s+t) - N_Y(s)$ given $N_Y(s) = m$ is the negative binomial$(m, e^{-t})$ distribution with generating function 
\begin{equation}
\label{negbingft}
\E \left( z ^{ N_Y(s+t) - N_Y(s) } \, \left| \, N_Y(s)   = m \right.\right) = \left( \frac{  e^{-t} }{ 1 - (1 - e^{-t})  z } \right)^m .
\end{equation}
\item There is the representation
\begin{equation}
\label{exprep}
Y_k = \sum_{i=1} ^k \frac{ \eps_{i} }{i} \qquad ( k \ge 0)
\end{equation}
for a sequence of i.i.d.\ standard exponential variables $\eps_i$, with $Y_0 = 0$.
\item 
For each fixed $k$ there is the equality of joint laws
\begin{equation}
\label{order-stats}
(\T_j, 0 \le j \le k) \ed (\eps_{n:n} - \eps_{n-j:n}, 0 \le j \le k ) \qquad \mbox{ for every } n \ge k,
\end{equation}
where $\eps_{0:n}:=0<\eps_{1:n} < \cdots <\eps_{n:n}$ is the sequence of order statistics of i.i.d.\ standard exponential variables $(\eps_i, 1 \le i \le n)$.
\item There is the representation
\begin{equation}
\label{kendallrep}
N_Y(t) = 1 + N_\gamma(  ( e^t - 1 ) \hateps_1  )
\end{equation}
where $N_\gamma(v):= \sum_{i = 1}^\infty \ind (\gamma_i \le v )$ is a rate $1$ Poisson process with arrival
times $\gamma_k = \sum_{i = 1}^k \tileps_i$ for a sequence of i.i.d.\  standard exponential variables $\tileps_i$,
and $\hateps_1$ a further standard exponential variable, independent of all the $\tileps_i$,  which may be identified as
\begin{equation}
\label{explimit}
\hateps_1 = \lim_{t \to \infty} e^{-t} N_Y(t) = \lim_{k \to \infty} k e^{- Y_k} \mbox{ almost surely}
\end{equation}
while
\begin{equation}
\label{gamvars}
\gamma_k = ( e^{Y_k} - 1 ) \hateps_1 \qquad ( k = 1,2 , \ldots ) .
\end{equation}
\item
There is the representation
\begin{equation}
\label{logrep}
Y_k = \log ( \hatgam_{k+1} /\hatgam_1 ) \qquad ( k = 0,1,2 , \ldots ) .
\end{equation}
where $\hatgam_{k+1} = \sum_{i = 1}^{k+1} \hateps_i$ for a sequence of i.i.d.\  standard exponential variables $\hateps_i$, 
which may be identified as $\hatgam_1 = \hateps_1$ as in \eqref{explimit} for $k = 0$, and $\hatgam_{k+1} = \hateps_1 + \gamma_k$ as in \eqref{gamvars} for $k \ge 1$.
\end{itemize}
\end{proposition}

The equivalence of the first two descriptions can be found in 
Feller \cite[\S XVII.3]{MR0228020}. 
The representation \eqref{exprep} of the pure birth process in terms of independent exponential holding times is well known,
as is the equivalence of \eqref{exprep} and the representation \eqref{order-stats} in terms of the order statistics of a sequence of i.i.d.\ exponential variables,  
which is due to Sukhatme \cite{sukhatme1937tests} 
and R\'enyi \cite{MR0061792}. 
The representation \eqref{kendallrep} is due to Kendall \cite[Theorem 1]{MR0198551}. 
See also \cite[p.~127, Theorem 2]{MR0373040}, 
\cite[Theorem 3.12]{kallenberg-rm},
\cite{MR0293732}, 
and 
\cite[Theorem 4.3]{MR1463943} 
for a more general characterization due to Lundberg \cite{lundberg1940random} 
of when a mixed Poisson process can be time-changed to a Markovian birth process with stationary birth rates.
See also \cite{MR579823}, 
\cite[Lemma 6]{MR1825154},  
\cite[p.~532]{MR950166}, 
\cite{MR531764}, 
\cite{MR1102879}  
for variations and applications of the representation \eqref{kendallrep} of the Yule process.
The identification \eqref{explimit} follows immediately from \eqref{kendallrep} and the strong law of large numbers
for the Poisson process $N_\gamma$. Then \eqref{gamvars} follows, because the random time $t = Y_k$ when $N_Y$ first reaches $k + 1$,
corresponds to the random time $v = ( e^{Y_k} - 1 ) \hateps_1 $ when $N_\gamma$ first reaches $k$, that is $v = \gamma_k$.
The final representation \eqref{logrep} is just a rearrangement of \eqref{gamvars}, with the indicated change of variables,
as given in \cite[(5.3)]{MR0287620}. 
The equivalence of the two very different looking representations  \eqref{exprep} and \eqref{logrep} of $(Y_1, Y_2, \ldots )$
can be quickly checked as follows.
By well known beta-gamma algebra (see, e.g.~\cite{MR1654531}), in the construction \eqref{logrep}
the random variables
$\hatgam_1/\hatgam_2, \hatgam_2/\hatgam_3, \ldots, \hatgam_{i}/\hatgam_{i+1}$ are independent, with 
$$
\hatgam_{i}/\hatgam_{i+1} \ed 1 - \beta_{1,i}  \qquad \mbox{ and } - \log(1 - \beta_{1,i} ) \ed \eps_i/i  
$$
as required. Tracing back through the definitions, this calculation also implies Kendall's representation \eqref{kendallrep}.

\section{Proofs}
\label{sec:proofs}

In a sample $X_1, \ldots, X_n$ from  $\Pb$ defined by the RAM \eqref{stickbreak},
the number of values $X_i$ strictly greater than $1$ has the mixed binomial$(n,p)$ distribution with $p$ assigned the distribution of $1-H$. That is
\begin{equation}
\label{qstartp}
\P \left( \sum_{i=1}^n \ind ( X_i > 1) = m \right) = q^*(n,m):= {n \choose m } \mum{n-m}{m}
\end{equation}
where $\mum{n-m}{m}:= \E H^{n-m} ( 1 - H)^m$.
This $q^*(\ell, m)$ is the transition probability function of each of the {\em tail count Markov chains} 
$\Qstar_{\bullet:n} : = (\Qstar_{k:n}, k= 0, 1, \ldots )$ which describe the sequences of tail counts obtained when $n$ balls are
distributed according to the RAM with i.i.d.\ factors distributed like $H$:
\begin{equation}
\label{qstartdef}
\Qstar_{k:n} := \sum_{i = 1}^ n \ind ( X_i > k )  \qquad ( k = 0,1, \ldots ) .
\end{equation}
So $\Qstar_{\bullet:n}$ is the Markov chain with transition matrix \eqref{qstartp} and initial state  $\Qstar_{0:n} = n$.
The $*$ is used here to match notation with the overview \cite[\S 2]{MR2735350},  
where $W:= 1 -H$ so that $\mum{n-m}{m}:=  \E ( 1-W)^{n-m} W ^m$,
and to distinguish $q^*$ from the decrement matrix $q$ that appears in
Gnedin \cite{MR2044594} 
and Gnedin and Pitman  \cite{MR2122798}.  
Those articles are concerned with the {\em composition of $n$} generated by the configuration of $n$ balls
in boxes, that is the list of counts of non-empty boxes. For that purpose, the chain with transition matrix $q^*$ is watched only when it moves. By renormalizing the
off-diagonal elements of $q^*$, the resulting
transition matrix $q$ with zero diagonal entries is given by
\begin{equation}
\label{qmn}
q(n,m):=    \frac{ q^*(n,m) \ind (n > m )} {1 - q^*(n,n)} = {n \choose m } \frac{ \mum{n-m}{m} \ind (n > m )} {1 - \mum{0}{n} }
\end{equation}
which explains the factor $1 - \mum{0}{n}$ in the denominator of formulas of 
\cite{MR2044594} 
and \cite[Example 2]{MR2122798}. 
Beware that the {\em decrement matrix} of 
\cite{MR2044594} 
and \cite[(11)]{MR2122798} 
is the matrix with entries $q(n,n-m)$ rather than $q(n,m)$ as in \eqref{qmn}. 

By definition,
\begin{equation}
\label{mNdef}
M_n:= \max_{1 \le i \le n} X_i = \min \{k : \Qstar_{k:n} = 0  \},
\end{equation}
is the almost surely finite random time at which the transient Markov chain $\Qstar_{\bullet:n}$ is absorbed in state $0$.
Theorem \ref{thm:main} is concerned with the time reversal of this Markov chain, with time counted back from the step before the absorption time, that is
\begin{equation}
\label{Qrev}
\hatQ_{k:n}:= \Qstar_{M_n - 1 - k :n} ,  \qquad ( k \ge 0  )
\end{equation}
with the convention that $\hatQ_{k:n}:= \infty$ for $k \ge M_n$, meaning that $\hatQ_{\bullet:n}$ reaches an absorbing state $\infty$ at time $M_n$, just as
$\Qstar_{\bullet:n}$ reaches its absorbing state $0$ at time $M_n$. The initial state of $\hatQ_{\bullet:n}$ is $\hatQ_{0:n}:= \Qstar_{M_n - 1:n} = N_{M_n:n}$, the number of balls in the rightmost
occupied box after $n$ balls have been thrown. Also, the last state of 
$\hatQ_{\bullet:n}$ before absorption at $\infty$ is $\hatQ_{M_n-1:n} = \Qstar_{0:n} = n$.
Hunt \cite{MR0123364} showed
that the time reversal of a Markov chain with stationary transition probabilities such as
\eqref{Qrev} 
is 
another Markov chain with stationary transition probabilities, say $\hatq _n(\ell,m)$, which  may be described as follows.
For $n, m \ge 1 $ let
\begin{equation}
\label{potent}
g_{m:n} := \sum_{k=0}^\infty \P ( \Qstar_{k:n} = m )  = \sum_{k=0}^\infty \P ( \hatQ_{k:n} = m )  
\end{equation}
be the potential function giving the expected number of times that either $\Qstar_{\bullet:n}$ or its time reversal $\hatQ_{\bullet:n}$ hits state $m$.  
Then for positive integers $m$ and $\ell $ the expected number of $m$ to $\ell$ transitions of
the original tail count chain $\Qstar_{\bullet:n}$ can be computed in two different ways as
\begin{equation}
\label{revtrans}
g_{m:n} \, q^*(m,\ell) = g_{\ell:n} \, \hatq_n(\ell,m)
\end{equation}
which rearranges to give
\begin{equation}
\label{hatqdef}
\hatq_n(\ell,m) =  \frac{ g_{m:n}}{g_{\ell:n}} q^*(m,\ell)  .
\end{equation}

A proof of Theorem \ref{thm:main} will combine the following two lemmas.

\begin{lemma}
\label{lmm:mcrev}
Let $\Qstar_{\bullet:n}$ be a sequence of Markov chains, each with the same transition probability function $q^*(m,\ell)$ on the set $\{0,1, 2, \ldots\}$,
such that\/ $\P(M_n < \infty ) = 1$ for every $n$, where $M_n$ is the time $\Qstar_{\bullet:n}$ first hits state $0$, assumed to be an absorbing state.
Suppose further that for every $m = 1,2, \ldots$
\begin{equation}
\label{limassms}
\lim_{n \to \infty} \P(\Qstar_{0:n} = m) = 0 \mbox{ and } \lim_{n \to \infty} g_{m:n} = g_m \mbox{ with } 0 < g_m < \infty \mbox{ and } \sum_{m=1}^\infty g_m q^*(m,0) = 1. 
\end{equation}
Let $\hatQ_{\bullet:n}$ be
the chain $\Qstar_{\bullet:n}$ reversed back from time $M_n-1$ as in \eqref{Qrev}. Then, as $n \to \infty$,  the finite-dimensional distributions of 
$\hatQ_{\bullet:n}$  converge to those of a limit process  $\hatQ_{\bullet:\infty}$, which is a transient Markov chain with state space $\{1,2, \ldots\}$, 
transition probability matrix $\hatq(\ell,m):= g_m \, q^*(m,\ell)/g_\ell$, initial probability distribution 
\begin{equation}
\label{hatinit}
\P(\hatQ_{0:\infty} = m ) =  g_m q^*(m,0)  \qquad ( m = 1,2, \ldots)
\end{equation}
and potential function $\sum_{k=0}^\infty \P( \hatQ_{k:\infty} = m ) = g_m$ for all $m = 1,2, \ldots$.
\end{lemma}
\begin{proof}
By decomposing over values of $M_n$, formula \eqref{hatinit} holds 
with $\hatQ_{0:n}$ instead of $\hatQ_{0:\infty}$ and $g_{m:n}$ instead of $g_m$ 
for each $n$. 
The assumptions \eqref{limassms} ensure that the initial state $\hatQ_{0:n}$ of $\hatQ_{\bullet:n}$ converges in distribution to $\hatQ_{0:\infty}$ with distribution \eqref{hatinit}. 
Together with the well known equation \eqref{potential} recalled in the Appendix, which is satisfied by the potential function
\begin{equation}
g_{\ell:n}  = \P(\Qstar_{0:n} = \ell) + \sum _{m = 1}^\infty g_{m:n} \, q^*(m,\ell)
\end{equation}
for every $n$, these assumptions imply that 
a stochastic matrix $\hatq(\ell,m)$ is obtained as the $n \to \infty$ limit of
$\hatq_n(\ell,m)$ in \eqref{hatqdef}. The probability of any particular finite sequence of  values of $(\hatQ_{i:n}, 0 \le i \le k)$, say
$(m_0, \ldots , m_k)$ is
$$
g_{m_0:n} \, q^*(m_0,0)  \, \prod_{i = 0}^{k-1} \hatq_n(m_i,m_{i+1}) \to  g_{m_0} \, q^*(m_0,0)  \prod_{i = 0}^{k-1} \hatq(m_i,m_{i+1}) \mbox{ as } n \to \infty
$$
and the conclusion follows.
\end{proof}

Lemma \ref{lmm:mcrev} also applies to the theory of random walks conditioned to
stay positive. For instance, for $\Qstar_{\bullet:n}$ a simple symmetric random walk with increments of $\pm 1$ stopped when it first reaches $0$,
the time reversed limit chain $\hatQ_{\bullet:\infty}$ is the random walk started at $1$ and conditioned never to hit $0$, as discussed in \cite{MR0375485}. 

That the assumptions of Lemma \ref{lmm:mcrev} are satisfied by the particular sequence of tail count Markov chains $\Qstar_{\bullet:n}$ derived from a
RAM is a consequence of the next Lemma. 

\begin{lemma}
\label{lmm:potent}
In the setting of Theorem \eqref{thm:main}, for a RAM such that $-\log(1-H)$ has a non-lattice distribution with finite mean $\mulog$, 
the limit of the potential function \eqref{potent} of the chain $\Qstar_{\bullet:n}$ as $n \to \infty$ is
\begin{equation}
\label{limpot}
\lim_{n \to \infty} g_{m:n} = \frac{ 1 }{ m \, \mulog  } .
\end{equation}
\end{lemma}
\begin{proof} 
From the transition matrix \eqref{qstartp}, given that the chain $\Qstar_{\bullet:n}$ ever hits state $m$, the expected number of times
it stays there is $1/(1-q^*(m,m))  = (1 - \mum{0}{m})^{-1}$. So 
$$
g_{m:n} = \P( \Qstar_{k:n} = m \mbox{ for some } k \ge 0 ) (1 - \mum{0}{m} )^{-1}
$$
and \eqref{limpot} is equivalent to the result of Gnedin \cite[Proposition 5]{MR2044594} 
that under the above assumptions
\begin{equation}
\label{limpotgn}
\lim_{n \to \infty} \P( \Qstar_{k:n} = m \mbox{ for some } k \ge 0 ) = \frac{ 1 - \mum{0}{m} }{ m \, \mulog  } .
\end{equation}
Gnedin's proof of \eqref{limpotgn} can be simplified to obtain  \eqref{limpot} directly  as follows. 
Let $S_k:= \sum_{i = 1}^k - \log ( 1 - H_i)$ so the $S_k$ with $S_0 = 0$ are the times of renewals in a  renewal point process on $[0,\infty)$,
with associated counts $N_S(s,t]:= \sum_{k=1}^\infty \ind ( S_k \in (s,t])$ for $0 \le s < t < \infty$. 
Note that this renewal point process has zero delay as in Section \ref{sec:sampingram}, rather than the stationary delay as in Corollary \ref{crl:yule}.
Let 
$
\eps_{1:n} < \cdots < \eps_{n:n}
$
be the order statistics of $n$ i.i.d.\ exponential variables $\eps_1, \ldots, \eps_n$, so that the tail count chains 
may be constructed according to Lemma \ref{lmm:pp} as
\begin{equation}
\label{qstareps}
\Qstar_{k:n} := \sum_{i = 1}^n \ind ( \eps_i > S_k ) = \sum_{i = 1}^n \ind ( \eps_{i:n} > S_k ).
\end{equation}
Then for $1 \le m < n$ and $k \ge 0$, there is the almost sure identity of events
\begin{equation}
\label{qstareps1}
(\Qstar_{k:n}  = m ) = ( \eps_{n-m:n} < S_k \le \eps_{n-m + 1:n})
\end{equation}
and hence as $n \to \infty$
\begin{equation}
g_{m:n}  := \E \sum_{k=0}^\infty \ind (\Qstar_{k:n}  = m ) =  \E N_S(\eps_{n-m:n} , \eps_{n-m + 1:n}] 
 \to 
\frac{1}{m \, \mulog} 
\end{equation}
by application of Blackwell's renewal theorem and the representation \eqref{order-stats} of exponential order statistics,
which together with \eqref{exprep} gives $\E (\eps_{n-m+1:n} - \eps_{n-m :n}) = 1/m$.
\end{proof}

\begin{proof}[Proof of Theorem \ref{thm:main}]
Lemma \ref{lmm:potent} gives the convergence of potential functions
required in \eqref{limassms} for application of 
the time reversal Lemma  \ref{lmm:mcrev} to conclude that $\hatQ_{\bullet:N}$ converges in distribution to $Q_{\bullet}:= \hatQ_{\bullet:\infty}$ as
in the Lemma.  Since $\Qstar_{0:n} = n$ the first condition in \eqref{limassms} is satisfied,
and so is the last condition, because switching the order of summation and integration shows that the limit distribution \eqref{entrancelaw} of $Q_0$ sums to $1$
for any distribution of $H$ with $\mulog < \infty$.
Formula \eqref{backfor} comes from combination of \eqref{qstartp} and \eqref{limpot}.
The joint convergence in distribution of $Q$, $N$ and $G$ sequences as in \eqref{nglims} follows easily from the
result for the $Q$ sequences, using the explicit formulas for $N$ and $G$ in terms of $Q$, both before and after passage to the limit,
which are also summarized in the claim.
\end{proof}

\begin{proof}[Proof of Corollary \ref{crl:branching}]
This is clear by inspection of formula \eqref{backfor} for the transition mechanism of $\Qb$.
\end{proof}

\begin{proof}[Proof of Corollary \ref{crl:yule}]
This follows by comparison of Corollary \ref{crl:branching}
with the transition mechanism of the Yule process described in Proposition \ref{prp:yule}.
A more insightful argument can be given as follows,
by further development of the argument around \eqref{qstareps}.
Blackwell's renewal theorem gives $\lim_{t \to \infty} \BE N_{S}(t, t + h ] = h/\mulog$, as well as
convergence in distribution as $t \to \infty$ of the shifted counting process $(N_S(t , t+h ], h \ge 0 )$ 
to that of the stationary renewal process
\begin{equation}
\label{stat1}
N_S^*(0,h] := \sum_{i = 0}^\infty \ind ( S^*_i \le h )  \qquad (h \ge 0 )
\end{equation}
where $(S^*_z, z \in \BZ)$ is the collection of points of the two-sided stationary renewal point process with spacing distribution that of $-\log(1-H)$ and
the stationary start $S^*_0$ assigned the distribution \eqref{sodist}.
Together with the well known reversibility property of $N_S^*$ and the Yule representation 
\eqref{order-stats} of the order statistics of exponential variables, 
this gives the description of Corollary \ref{crl:yule}
for the limit in distribution of the $G$-sequence in Theorem \ref{thm:main}.
The corresponding results for the $N$- and $Q$-sequences follow easily using the duality between the $N$- and $G$-sequences
described in the introduction, in terms of interleaving the points of the two point processes.
\end{proof}

The convergence part of the above argument glosses over a few details, but they are technically easier than those dealt with in
\cite{MR2508790} 
in a quite similar proof.
The present method of looking down from the 
maximum of exponential order statistics means it is only necessary to deal with convergence in distribution of counts of a one-sided  renewal process,
looking back from a large random time, which is easily done. Moreover, the limit comes out simply expressed in terms of just one side of
the two-sided stationary renewal process, rather than involving both sides of the two-sided process, as in \cite{MR2508790}. 

\begin{proof}[Alternative proof of Corollary \ref{crl:yule}]
\label{sec:secondproof}
As indicated already in the introduction, it is quite easy to pass algebraically between the alternate formulations of
Theorems \ref{thm:gir} and \ref{thm:main}. Gnedin, Iksanov and Roesler already gave an interpretation of their
Theorem \ref{thm:gir} in terms of an interleaving of points of two point processes. However, their point process
interpretation  of the limiting $\Nb$ sequence is not obviously equivalent to the much simpler point process description of Corollary \ref{crl:yule}.
So it seems worth indicating how Corollary \ref{crl:yule} can nonetheless be derived from the construction of
$\Nb$ given in \cite{MR2508790}.
In 
\cite[Theorem 2.1]{MR2508790} and 
the proof of 
\cite[Lemma 3.4]{MR2508790} 
the limit sequence $\Nb$ was constructed by first exponentiating the two-sided stationary renewal process $S^*_{z}, z \in \BZ$ on the logarithmic scale to create a scale-invariant 
(self-similar) point process $\xi_z^*:= \exp(S_z^*), z \in \BZ$ with
\begin{equation}
0 < \cdots < \xi_{-2}^* < \xi_{-1}^* < 1 < \xi_0^* < \xi_1^* 
\end{equation}
together with an independent homogeneous Poisson rate $1$ sampling process, with points at say $0 < \gamma_1 < \gamma_2 < \cdots$ with
independent standard exponential spacings. The informal description from 
\cite{MR2508790} 
is to regard the $\gamma_k$ as the locations of balls thrown into
a doubly infinite array of boxes $(\xi^* _{z}, \xi^* _{z+1})$ indexed by $z \in \BZ$, then let $(N_0, N_1, \ldots)$ be the counts of balls in consecutive
boxes, starting with $N_0$ the count of balls in the first occupied box, that is the box $( \xi^*_Z, \xi^*_{Z+1})$ with $\xi^*_Z < \gamma_1 < \xi^*_{Z+1}$
for the random integer index $Z$  determined by this condition, ignoring the event of probability zero that there might be any ties between the $\xi_z$ and $\gamma_k$.
That is to say
\begin{equation}
\label{nfromxi}
N_i := \sum_{k=1}^\infty \ind ( \xi_{i} < \gamma_k < \xi_{i+1} ) \mbox{ where } \xi_{i} := \xi^*_{Z+i}\qquad(i=0,1,2,\dots).
\end{equation}
As argued in the proof of 
\cite[Lemma 3.4]{MR2508790} 
the $\xi_k, k \ge 1$ may be constructed directly as
\begin{equation}
\label{xik}
\xi_k = \frac{ \gamma_1 }{W_0} \prod_{i=1}^{k-1} (1 - H_i)^{-1}
\end{equation}
where the $H_i$ are i.i.d.\ copies of the basic hazard variable, $1/W_0$ with $1/W_0 \ed \exp(S^*_0)$ is a further independent  multiplicative version of the stationary delay variable, 
and the Poisson points $\gamma_k$ are assumed to be independent of all these variables.  Note that the precise value of $\xi_0$ is of no importance in \eqref{nfromxi}, except for the
fact that $\xi_0 < \gamma_1 < \xi_1$, so that \eqref{nfromxi} for $i=0$ reduces to 
\begin{equation}
N_0 := 1 + \sum_{k=2}^\infty \ind ( \gamma_k < \xi_{1} ) .
\end{equation}
This formula states that count of balls in the leftmost occupied box is $1$ for $\gamma_1$ plus however many following $\gamma_k$ are to the left of the first split $\xi_1$ between boxes.
This construction is made difficult by the way the definition of the break points between boxes 
involves one of the Poisson sample points $\gamma_1$, which then affects the distribution of the other Poisson sample points which are all conditioned to  be greater than $\gamma_1$.
However, the construction can be simplified by remarking that the dual sequence $\Gb$ counting box dividers between balls is
\begin{align}
\label{gidef}
\notag 
G_i := \smash[b]{\sum_{k=1}^\infty} \ind ( \gamma_i < \xi_k < \gamma_{i+1} ) &= N^*_S ( \log \gamma_i, \log \gamma_{i+1}) \\
\notag						       &\ed N^*_S ( \log (\gamma_i/\gamma_1), \log (\gamma_{i+1}/\gamma_1))  \mbox{ by stationarity of } N^*_S\\
							       &\ed N^*_S ( \T_{i-1}, \T_i )
\end{align}
according to the representation \eqref{logrep} of the birth times $\T_i$ of a Yule process.
Moreover, it is clear that this identity in distribution holds jointly as $i$ varies.
Thus the point process description of $\Nb$ in
\cite{MR2508790}, 
implies the characterization of $\Gb$ provided in Corollary \ref{crl:yule}, hence also the
dual description of $\Nb$ in terms of the Yule process.
\end{proof}


The proof of Corollaries \ref{crl:gth} and \ref{crl:g} is based on theory of increasing Markov chains recalled in the Appendix.

\begin{proof}[Proof of Corollary \ref{crl:g}]
This follows easily from Proposition \ref{prp:transmc}.
\end{proof}

\begin{proof}[Proof of Corollary \ref{crl:gth}]
That the limiting gaps $\Gb$ for the GEM$(0,\theta)$ model are independent with the indicated geometric distributions can be
read from the exact result \eqref{gapstheta} for a sample of size $n$, as explained in the introduction.
As for the converse assertions, suppose first that the limiting gaps $\Gb$ are geometrically distributed.
Then 
$$
h_j:= \P(G_j \ge 1) = 1 - p_{j,j} = 1 - \mum{0}{j}.
$$
On the other  hand, \eqref{limpotgn} gives $h_j = (1 - \mum{0}{j})/ ( j \mulog)$.
Equating these two expressions for $h_j$ gives  for $j = 1,2, \ldots$.
$$
\E (1 - H)^j = \mum{0}{j} = \frac{ \theta }{ \theta + j } = \E (1 - \beta_{1,\theta})^j
$$
where $\theta:= 1/\mulog$ and $\beta_{1,\theta}$ has the beta$(1,\theta)$ distribution with
density $\theta(1-u)^{\theta - 1}$ at $ 0 < u < 1$. Hence $H \ed \beta_{1,\theta}$.

Suppose next that 
 $\P(\hatG_{1:n}\ge 1)$ does not depend on $n$. Given $\Pb$, the conditional 
probability 
$$
\P (\hatG_{1:n}\ge 1|\Pb)=n\sum_{m=1}^{\infty} P_m(P_1+\dots+P_{m-1})^{n-1}
$$
by the decomposition over the maximal value $m$ which appears only once iff $\hatG_{1:n}\ge 1$. This gives for a RAM with i.i.d.\ factors $H_1,H_2,\dots $
\begin{align}
\notag 
\P(\hatG_{1:n}\ge 1)&{}=n \sum_{m=1}^\infty \E\Bigl[ H_m \prod_{i=1}^{m-1}(1-H_i)\Bigl(1-\prod_{i=1}^{m-1}(1-H_i)\Bigr)^{n-1}\Bigr]\\
\notag 
&{}=n\mum{1}{0}\sum_{m=1}^{\infty}\sum_{k=0}^{n-1}\binom{n-1}{k}(-1)^k(\mum{0}{k+1})^{m-1}\\
&{}=n\mum{1}{0}\sum_{k=0}^{n-1}\binom{n-1}{k}(-1)^{k} \frac{1}{1-\mum{0}{k+1}}.
\end{align}
By a simple algebra, the equality $\P(\hatG_{1:n} \ge 1)=\P(\hatG_{1:n+1} \ge 1)$ translates into the following
recursion
\begin{equation}\label{rec}
\frac{n+1}{ 1- \mum{0}{n+1}}=\sum_{k=0}^{n-1}\binom{n}{k}(-1)^{n-k-1}\frac{k+1}{1 - \mum{0}{k+1}}
\end{equation}
which allows to define $\mum{0}{n}$ one by one starting from arbitrary $\mum{0}{1}\in(0,1)$. Since the moments $\mum{0}{n}$ determine the distribution
of $H$, and in GEM$(0,\theta)$ model 
satisfy \eqref{rec} as follows from \eqref{gapstheta} or can be easily verified directly using \eqref{thetavals}, the claim follows.

Finally, suppose that the limiting gaps $G_j$ are independent as $j$ varies.
According to Proposition \ref{prp:rec}, the jumping chain derived from $\Qb$ is the strict record chain
derived from the initial distribution $\init{j}:= \mum{j}{0}/(j \mulog)$. This implies that
the transition matrix $p$ of $\Qb$ must satisfy the identity
\begin{equation}
\label{pmmnn}
\frac{ p_{m,n} }{ 1 - p_{m,m} } = \frac{ \init{n} }{ 1 - \init{1} - \cdots - \init{m} } \mbox{ for } 1 \le m \le n .
\end{equation}
Plugging in the formula \eqref{backfor} for $p$, this becomes
\begin{equation}
\label{mulogeq}
\frac{ { n-1 \choose m-1} \mum{n-m}{m} } { 1 - \mum{0}{m} } = \frac{ \mum{n}{0}/n } {\mulog - \mum{1}{0}/1 - \cdots - \mum{m}{0}/m  } .
\end{equation}
Consider now a distribution of $H$ on $(0,1)$ with $\E(H) = \mu_1$ and $\E - \log(1-H) = \mulog$ for some specified values
of $\mu_1$ and $\mulog$ subject to $0 < \mu_1 < \mulog < \infty$. 
Let $\theta:= \mu_1 /( \mulog - \mu_1)$ 
and
$\mu_n:= \mum{n}{0} = \E H^n$.
Then \eqref{mulogeq} for $m = 1$
yields easily
$$
\mu_n = \frac{n \mu_{n-1} }{\theta + n } \qquad ( n \ge 2 )
$$
and hence by induction on $n$
$$
\mu_n = ( 1 + \theta ) \mu_1  \frac{ (1)_n } { (1 + \theta)_n }  = ( 1 + \theta ) \mu_1  \E \beta_{1,\theta}^n
\qquad ( n \ge 1) .
$$
Thus $\E H^n = c \E \beta_{1,\theta}^n$ for every $n \ge 1$ for a constant factor $c:= (1 + \theta) \mu_1$.
To complete the argument, it only remains to show that $c = 1$, which achieved  by the following lemma.
\end{proof}

\begin{lemma}\label{lmm:thoH}
Let $\tilH$ and $H$ be two random variables with 
\begin{gather}
\label{in01}
\P( 0 < \tilH < 1 ) = \P( 0 < H < 1 ) = 1 \qquad \mbox{ and}\\
\label{feq}
\E \tilH^n   = c \E H^n  \mbox{ for some constant } c  \mbox{ and  all } n = 1,2, \ldots.
\end{gather}
 Then $c = 1$ and $\tilH \ed H$.
\end{lemma}
\begin{proof}
By switching the roles of $H$ and $\tilH$ if $c > 1$, it is enough to consider the case when $0 < c < 1$.
Let $B_c$  denote a Bernoulli$(c)$ variable independent of $H$. Then 
$$
\E  ( B_c H )^n =  
\E  ( B_c ^n H ^n)  = 
\E  ( B_c H ^n)  = \E  ( B_c ) \E H ^n  = c\,  \E H^n \qquad (n \ge 1 ).
$$
Because a distribution on $[0,1]$ is uniquely determined by its moments, it follows from \eqref{feq} that
$\tilH \ed B_c H$, the distribution of which has an atom of magnitude $1-c >0$ at $0$. 
Thus \eqref{feq} for $0 < c < 1$ violates \eqref{in01}, and the conclusion follows.
\end{proof}

{\bf Remark.}
An alternative proof that the limiting gaps $\Gb$ derived from GEM$(0,\theta)$ are both independent and
geometrically distributed is obtained via Proposition \ref{prp:rec}. According to that proposition, it is enough
to check that the transition matrix $p$ of $\Qb$ satisfies \eqref{pmmnn} for all $1 \le m \le n$,
that is to check \eqref{mulogeq} for the various moments derived from $H$ with beta$(1,\theta)$ distribution.
But using the evaluations \eqref{muij} of beta$(1,\theta)$ moments,
after cancelling common factors, the identity \eqref{mulogeq}  reduces to  the summation formula
\begin{equation}
\label{theta:msum}
\sum_{i=1}^m \frac{ ( 1)_{i-1} \, \theta  }{ ( 1 + \theta )_{i} } = 1 - \frac{ ( 1)_m }{ ( 1 + \theta)_m }  \qquad ( m = 1,2, \ldots).
\end{equation}
This is easily checked by induction on $m$, along with its
probabilistic interpretation
\begin{equation}
\label{prob:theta:msum}
\P( Q_0 > m ) \eqth \frac{ ( 1)_m }{ ( 1 + \theta)_m }  \qquad ( m = 1,2, \ldots).
\end{equation}
The limit case
\begin{equation}
\label{theta:infsum}
\sum_{i=1}^\infty \frac{ ( 1)_{i-1} \theta }{ ( 1 + \theta )_{i} } =  1  \qquad ( \theta > 0 )
\end{equation}
is the evaluation of $  _2F_1(1,1;2 +\theta;z)\,\theta /(1 + \theta)$ at $z = 1$ using Gauss's hypergeometric summation formula
$_2F_1(a,b;c;1) = \frac{\Gamma(c)\Gamma(c-a-b)}{\Gamma(c-a)\Gamma(c-b)}$ $( \operatorname{Re}(c-b-a)>0)$,
for $a=b=1$ and $c = 2 + \theta$.

\section{Further Remarks}

\subsection{Checking the entrance law}
Implicit in Theorem \ref{thm:main} is the fact that for any distribution of $H$ on $(0,1)$ with
$\mulog:= \E - \log(1 - H) < \infty$ and $-\log(1-H)$ non-lattice, that for the Markov chain $\Qb$
with mixed negative binomial transition matrix \eqref{backfor} derived from moments of $H$, with 
initial distribution
$$
p_{0,n}:= \frac{ \mum{n}{0}  }{ n\, \mulog }
$$
as in \eqref{entrancelaw},
the hitting probabilities $h_n:= \P(Q_k = n \mbox{ for some } k \ge 0 )$ are given by the formula
$$
h_n = \frac{ 1 - \mum{0}{n} } { n \, \mulog }  \qquad ( n \ge 1 ) .
$$
In fact, this is the case whenever $\mulog < \infty$, even if $-\log(1-H)$ has a lattice distribution. For, by the theory of increasing
 Markov chains recalled in 
Proposition \ref{prp:transmc},
the sequence $h_n$ is the unique solution of the system of last exit equations 
$$
h_n = p_{0,n} + \sum_{m = 1}^{n-1} \frac{ h_m \, p_{m,n} } { 1 - p_{m,m} }  \qquad (n = 1,2, \ldots )
$$
which for the particular mixed negative binomial transition matrix \eqref{backfor} expands to 
$$
\frac{ 1 - \mum{0}{n} } { n \mulog }   = \frac{ \mum{n}{0}  }{ n \mulog } + \sum_{m = 1}^{n-1} \frac{ 1 } { m \mulog } {n-1 \choose m-1} \mum {n-m}{m} .
$$
After writing ${n-1 \choose m-1}  = \frac{m }{n} { n \choose m }$, cancelling common factors, and some rearrangement of terms, this reduces to
$$
1 = \sum_{m=0}^n { n \choose m } \mum{n-m}{m} = \E [ H + (1-H) ] ^n
$$
which is true for any distribution of $H$ on $[0,1]$ by the binomial theorem.

\subsection{Some checks for $N_0$ and $N_1$}
According to Corollary \ref{crl:yule}, the distribution of $N_0 = Q_0$, representing the limit distribution of the number of balls $N_{M_n:n}$ in the last box as $n \to \infty$,
and the probability of corresponding events in terms of the limiting gaps $G_i$, 
can be evaluated as
\begin{equation}
\label{nodiost}
\BP(N_0 > k ) = \BP(G_i = 0, 1 \le i \le k) = \BP[ N^*_S(Y_k) = 0 ] = \BP(Y_k < S^*_0)  
\end{equation}
for $k = 1,2, \ldots$.
This probability can be computed by conditioning on $S^*_0$, and using the representation $Y_k \ed \max_{1 \le i \le k} \eps_i$ of 
\eqref{order-stats},
which makes $\P(Y_k \le s ) = (1 - e^{-s})^k$. Thus
\begin{align}
\label{complemk}
\notag
\BP(N_0 > k ) &=  \frac{1} {\mulog} \int_0^\infty \BP( - \log ( 1 - H) > s ) (1-e^{-s})^k \, ds \\
              &= \frac{1} {\mulog} \int_0^1 \BP( H > u ) \frac{ u^k }{1-u }  \,du      .
\end{align}
by the change of variables $u = 1- e^{-s},  ds = du/(1-u)$.
In the special case of GEM$(0,\theta)$, with $\mulog \eqth  \theta^{-1}$ and $\P(H>u) \eqth  (1-u)^\theta$, this gives easily
\begin{equation}
\BP(N_0 > k ) \eqth \theta \int_0^1 u^k (1-u)^{\theta -1} \,du   \eqth  \frac{(1)_k}{ (1 + \theta)_k }  .
\end{equation}
But for the general case it is easier to work with the complementary event. The same change of variables gives
\begin{align}
\notag 
\BP(N_0 \le k ) &=  \frac{1} {\mulog} \int_0^\infty \BP( - \log ( 1 - H) > s ) ( 1 - (1-e^{-s})^k ) \, ds \\
\notag &=  \frac{1} {\mulog} \int_0^1 \BP( H > u ) \frac{  1 - u^k }{ 1-u } \, du  \\
\notag &=  \frac{1} {\mulog} \int_0^1 \BP( H > u ) ( 1 + u + \cdots + u^{k-1} ) \, du \\
&=  \frac{1} {\mulog} \left( \frac{ \E H }{1}  + \frac{\E H^2 }{2} + \cdots \frac{ \E H^k } {k} \right) 
\end{align}
in agreement with formula \eqref{entrancelaw} for $\P(N_0 = k) = \P(Q_0 = k)$.

Next, consider the joint limit law of $N_0$ and $N_1$ for a general RAM, given $N_0:= N_Y(S^*_0) = m \ge 1$ say, meaning that the event $Y_{m-1} < S^*_0 < Y_{m}$
has occurred.
The case $m = 1$ is easiest. Then, as in the previous computation, except that the increment $N_Y(S^*_1) - N_Y(S^*_0)$ involves $S^*_1 - S^*_0 \ed - \log(1-H)$ instead of the
stationary delay, for $k \ge 1$
\begin{equation}
\label{renew}
\P(N_1 \ge k \giv N_0 = 1 ) = \int_0^\infty \BP ( - \log ( 1 - H) \in ds ) ( 1 - e^{-s} ) ^k  = \BE H^k
\end{equation}
by the now familiar change of variable. That is to say, the distribution of $N_1 $ given $N_0 = 1$ is mixed geometric$(1-H)$.
Given $N_0 = m$ instead of $N_0 = 1$, the situation is similar, except that there are $m$ independent lines of descent in the Yule process, 
and the total number of births in the Yule process in all $m$ lines of descent must be counted before the next renewal time. 
Now the mixing variable $1-H = e^{-(S^*_1-S^*_0)}$ is the same for all lines of descent, so conditional on $H$ and $N_0 = m$ the distribution of
$N_1$ is the sum of $m$ independent geometric$(1-H)$ variables. That is to say, $N_1$ given $N_0 = m$ has the mixed negative binomial$(m,1-H)$
distribution, as in Corollary \ref{crl:branching}.

\subsection{Results when $\mulog = \infty$}
It is known that 
\begin{equation}
\label{muinf}
\mbox{ {\em if $\mulog = \infty$ then $\hatG_{i:n}$ converges in probability to $0$ for every $i$,}}
\end{equation}
corresponding to a piling up of values at the sample maximum $M_n$, so the 
number $\hatN_{0:n}:= N_{M_n:n}$ of ties with the maximum converges in probability to $\infty$. 
See \cite{MR2538083} 
\cite{MR3176494} 
\cite{MR2926172} 
\cite{MR3416065} 
for further treatment of limit theorems in this case.

\subsection{Limits for small counts}
The simple formula \eqref{valposg} for the limiting mean $\E G_j$ is a companion of similar
limit results for counts $K_{j:n}:= \sum_{b=1}^{M_n}\ind(N_b = j) = \sum_{i= 0}^{M_n-1} \ind(\hatN_i = j)$, which were derived in 
\cite[Lemma 3.2 and Theorem 3.4] {MR2508790} 
and are checked again here:
\begin{align}
\lim_{n \to \infty} \BE K_{j:n} & = \BE K_j = \frac{1}{j \, \mulog } \qquad \mbox{ for } j = 1,2, \ldots, \\
\lim_{n \to \infty} \BE K_{0:n} &= \BE K_0  = \frac{ \BE ( - \log H ) } {\mulog}
\end{align}
where  the limit variables $K_j$  are encoded in $\Nb$ as
\begin{equation}
K_j :=  \sum_{i = 0}^\infty \ind ( N_i = j)  \qquad ( j = 0,1, \ldots).
\end{equation}
It is known  \cite{MR1177897} 
that in the GEM$(0,\theta)$ model the $K_j$ for $j \ge 1$ are independent Poisson variables with means $\theta/j$ for $j = 1,2, \ldots$. 
It would be interesting to have converses of these properties, similar to Corollary \ref{crl:gth}.
It is also known \cite{MR3176494} 
that for a general RAM subject to \eqref{nonlat} the distribution of $K_0$ is always mixed Poisson, 
that is $K_0 \ed N(\tau)$ for $N$ a rate one Poisson process independent of some non-negative random variable $\tau$, with $\tau \eqth \theta | \log \beta_{1,\theta}|$ in the GEM$(0,\theta)$ case
with $H \eqth \beta_{1,\theta}$.
\subsection{Examples}
\label{sec:examples}

Even without detailing the bijection between the limiting $N$ and $G$ sequences, it is easy to express the limits in distribution of various statistics of interest in 
terms of either of these sequences.
Most obviously, the number $L_n$ of ties with the maximum and the number $K_{0:n}$ of missing values are encoded in the counts and gaps as 
\begin{equation}
\label{biv1}
L_n = \min \{ i : G_{i:n} > 0 \}  = N_{M_n:n} \mbox{ and } K_{0:n} = \sum_{i=1}^n (G_{i:n} - 1)_+  = \sum_{r = 0}^{M_n-1} \ind (N_{M_n - r:n} = 0)
\end{equation}
and it follows from Theorem \ref{thm:main} that
\begin{equation}
\label{biv2}
(L_n, K_{0:n}) \convd \left( \min \{ i : G_{i} > 0 \}  , \sum_{i=1}^\infty (G_{i} - 1)_+  \right) = \left(N_0, \sum_{r= 0}^\infty \ind (N_r = 0 ) \right) .
\end{equation}
Here the convergence in distribution of $L_n$ follows immediately from either of the sequence limit theorems. 
The convergence in distribution of $L_n$ to $N_0$, and the description \eqref{entrancelaw} of its limit law, was given in \cite[Theorem 2.1]{MR2508790}, as well 
as in \cite[Theorem 2.4]{MR2538083}.  
The result for $K_{0:n}$ takes more work, because it is not
immediately obvious that its limit can be read as indicated just from convergence of  finite dimensional distributions in \eqref{nglims}.
Also, according to \cite[Proposition 3.3]{MR2508790}, 
unlike the case
for $L_n$, the limit of $K_{0:n}$ can be infinite almost surely, and is so iff $\E[ - \log H ] = \infty$.  These issues were taken care of in  
\cite{MR2508790} 
for the representation  in terms of the $N_i$, and a similar discussion can be provided for the representation above in terms of the $G_i$.
This $G$-representation of $K_{0:n}$ and its limit in distribution can also be seen in the special case of GEM$(0,\theta)$ in \cite[(19)]{MR2538083}. 

\appendix
\section{Increasing Markov chains}
Consider a Markov chain $\Qb$ with state space $\{1,2,\dots\}$ and stationary transition matrix $(p_{i,j})$. 
Call $\Qb$ {\em weakly increasing} if $p_{i,j} = 0$ for $j < i$, so that
$$
\P( Q_0 \le Q_1 \le \cdots ) = 1
$$
and call $\Qb$ {\em strictly increasing} if $p_{i,j} = 0$ for $j \le i$, so that
$$
\P( Q_0 <  Q_1 < \cdots ) = 1 .
$$

The following proposition recalls some standard facts about such Markov chains \cite{MR1600720}.

\begin{proposition}
\label{prp:transmc}
Let $\Qb$ be a weakly increasing Markov chain with state space $\{1, 2, \ldots \}$, initial distribution
\begin{equation}
\init{j}:= \P(Q_0 = j)
\end{equation}
and stationary transition matrix $(p_{i,j})$. Suppose that $\Qb$ has no absorbing states, that is $p_{j,j}<1$ for all $j$.
Let $G_j:= \sum_{k=0}^\infty \ind ( Q_k = j)$ be the number of visits to state $j$. Then
\begin{itemize}
\item[(i)]
$G_j$ has the zero-modified geometric distribution
with parameters $(h_j, 1 - p_{j,j} )$, meaning that
\begin{equation}
\label{zmod}
\P(G_j \ge k) =  h_j  p_{j,j}^{k-1}   \qquad ( k = 1,2, \ldots)
\end{equation}
where
\begin{equation}
\label{hdef}
h_j: = \P(G_j \ge 1) = \P(Q_k = j \mbox{ for some } k \ge 0 ); 
\end{equation}
\item[(ii)] the conditional distribution of $G_j$ given $G_j \ge 1$ is the
geometric$( 1 - p_{j,j})$ distribution on $\{1,2, \ldots\}$, so
\begin{equation}
\label{zmodplus}
g_j:= \E G_j  = \sum_{k=1}^\infty \P(G_j \ge k ) = \frac{ h_j }{ 1 - p_{j,j}} ;
\end{equation}
\item[(iii)]
the potential sequence $(g_j, j \ge 1)$ is the unique non-negative solution of
the system of equations
\begin{equation}\label{potential}
g_j = \init{j} + \sum_{i = 1}^\infty g_i p_{i,j} \qquad ( j \ge 1);
\end{equation}
\item[(iv)]
equivalently,
the sequence of hitting probabilities $(h_j, j \ge 1)$ is the unique non-negative solution of
the system of equations
\begin{equation}\label{hitting}
h_j = \init{j} + \sum_{i \ne j }^\infty \frac{ h_i p_{i,j} }{1 - p_{i,i}} \qquad ( j \ge 1) .
\end{equation}
\end{itemize}
\end{proposition}

\begin{proof}
For transient Markov chains it is well known (e.g.~\cite[Lemma 1.5.2]{MR1600720}) that claim (ii) holds with the probability to ever return back from state $j$ instead
of $p_{j,j}$; that this probability is $p_{j,j}$ for weakly increasing Markov chains is obvious. Then \eqref{zmod} follows directly from the definitions
and implies \eqref{zmodplus} by summation.
The mean number of visits to state $j$ satisfies \eqref{potential} because one can either start with $Q_0=j$, which happens with probability $p_{0,j}$, or
come to $j$ from some state $i$, and the mean number of such transitions is $g_ip_{i,j}$, so equation \eqref{potential} holds for arbitrary Markov chains 
with countable state space. The uniqueness of its solution is evident for weakly increasing Markov chains with $p_{i,j}=0$ for $j<i$. Finally, 
the change of variables \eqref{zmodplus} implies \eqref{hitting} by rearranging terms.
\end{proof}

Let $\initnoarg := (\init{j}, j = 1,2, \ldots)$ be a fixed probability distribution with $\init{j} >0$ for infinitely many $j$.
Then there is the following well known construction of two Markov chains derived from $\initnoarg$.
The {\em weak record chain} generated by $\initnoarg$ is the sequence of weak upper record values
generated by an i.i.d.\ sequence with distribution $\initnoarg$.
The {\em strict record chain} generated by $\initnoarg$ is similarly generated sequence of strict upper record values,
which is the sequence of values of the weak record chain watched only when there is a change of value.
It is known 
\cite{MR0362457} 
and easily verified that the weak record chain is Markov with initial distribution $\initnoarg$ and transition matrix
\begin{equation}
\label{rhoj}
p_{i,j}^{\le} = \frac{ \init{j} \ind ( i \le j ) } { \init{i} + \init{i+1} + \cdots}
\end{equation}
while the strict record chain is Markov with initial distribution $\initnoarg$ and transition matrix
\begin{equation}
\label{rhojstrict}
p_{i,j}^{<} = \frac{ \init{j} \ind ( i < j ) } { \init{i+1} + \init{i+2} + \cdots}\,.
\end{equation}
It is also well known \cite[Theorems 16.1, 16.8]{MR1791071} that each of these Markov chains derived from $\initnoarg$ has the special property that
\begin{equation}\label{defGj}
\mbox{ \em the counts $G_j:= \sum_{k=0}^\infty \ind ( Q_k = j)$ are independent as $j$ varies }
\end{equation}
where the distribution of $G_j$ is 
\begin{itemize}
\item geometric$(1 - p_{j,j}^{\le})$ for the weak record chain, and 
\item Bernoulli$(p_{j,j}^{\le})$ for the strict record chain.
\end{itemize}

The following proposition offers a converse:
\begin{proposition}
\label{prp:rec}
Let $\Qb$ be a weakly increasing Markov chain with transition matrix $(p_{i,j})$ on the set of positive integers,
with initial distribution $\initnoarg$ such that $\init{j} >0$ for infinitely many $j$.
Suppose that 
the occupation times $G_j$ defined by \eqref{defGj} satisfy\/ $\P(G_j < \infty) = 1$ and\/ $\P(G_j = 0) < 1$ for every $j$,
and 
are independent. Then
\begin{equation}
\label{hledef}
h_j : = \P( G_j \ge 1 ) = p_{j,j}^{\le}
\end{equation}
is derived from $\initnoarg$ as in \eqref{rhoj}, which implies
\begin{equation}
\label{infprod}
\prod_{j=1}^\infty (1 - h_j ) \, = \, 0.
\end{equation}
Moreover, the transition matrix $(p_{i,j})$ of $\Qb$ is of the form
\begin{align}
\label{pijfromii}
p_{i,j} &{}= \ind ( j = i ) p_{i,i} + \ind  ( j>i ) ( 1- p_{i,i})  h_j \prod_{ k = i+1}^{j-1} ( 1 - h_k)\\
\label{pijfrominit}
&{}=\ind ( j =i ) p_{i,i}+\ind ( j>i )(1-p_{i,i})p_{i,j}^{<}
\end{align}
for some arbitrary sequence of self-transition probabilities $p_{i,i}$ with $p_{i,i} < 1$, where $p_{i,j}^{<}$ is defined by \eqref{rhojstrict}.
In particular, continuing to assume that the $G_j$ are independent, the following three conditions are equivalent:
\begin{itemize}
\item $p_{i,i} = h_i$ for all $i$;
\item each of the $G_j$'s is geometrically distributed on $\{0,1,2,\ldots\}$;
\item $\Qb$ is the weak record chain derived from $\initnoarg$;
\end{itemize}
and so too are the following three conditions
\begin{itemize}
\item $p_{i,i} = 0$ for all $i$;
\item each of the $G_j$'s has a Bernoulli distribution on $\{0,1\}$;
\item $\Qb$ is the strict record chain derived from $\initnoarg$.
\end{itemize}
In any case, whatever the self-transition probabilities $p_{i,i}$,
the chain $\Qb$ watched only when it changes state is a copy of the strict record chain derived from $\initnoarg$.
\end{proposition}
\begin{proof}
By the general theory of transient Markov chains reviewed in Proposition \ref{prp:transmc},
the distribution of $\Gb$ is that of a sequence
of random variables with zero-modified geometric$(h_j, 1 - p_{j,j})$ distributions, as displayed
in \eqref{zmod}. The assumed independence of the counts $\Gb$ implies that for $j \ge 1$
\begin{align} 
\P(Q_0 > j ) &{}= \P( G_1 = 0, \ldots, G_{j-1} = 0, G_j = 0 )\\
&{} = \prod_{i=1}^{j} (1 - h_i) .
\end{align}
The assumption that $\P(Q_0  < \infty)$ then gives \eqref{infprod}.
 On the other hand, it is clear that the entire path of the weakly increasing Markov chain $\Qb$ can be recovered
with probability one from its sequence of occupation times $\Gb$. So the transition matrix of $\Qb$ is a function of
the two sequences $(h_j)$ and $(p_{j,j})$ which should be chosen consistently with the initial distribution to make counts $\Gb$ independent. 
If the chain jumps from $i$ to $j>i$ it means that it does not stay at $i$, which happens with probability $1-p_{i,i}$, and
given this event that $G_{i+1}=\dots=G_{j-1}=0$ and $G_j>0$, so \eqref{pijfromii} follows from independence of the counts $\Gb$.
Then a tedious but straightforward calculation shows by induction on $j$ that $h_j=p_{j,j}^{\le}$ provides a solution
for \eqref{hitting}, so \eqref{hledef} holds
and gives \eqref{pijfrominit} after some algebra. Finally, both assertions about equivalence of three conditions and the last
assertion of the proposition now follow easily from comparison of \eqref{zmod}, \eqref{rhoj} and \eqref{rhojstrict}.
\end{proof}



\providecommand{\bysame}{\leavevmode\hbox to3em{\hrulefill}\thinspace}
\providecommand{\MR}{\relax\ifhmode\unskip\space\fi MR }
\providecommand{\MRhref}[2]{%
  \href{http://www.ams.org/mathscinet-getitem?mr=#1}{#2}
}
\providecommand{\href}[2]{#2}

\bibliographystyle{plain}
\bibliography{rambib,../gemmax,../0ergodic}

\end{document}